\documentclass[reqno,12pt]{amsart}

\usepackage{amsmath,amssymb,amsthm,graphicx,mathrsfs,url}
\usepackage[usenames,dvipsnames]{color}
\usepackage[colorlinks=true,linkcolor=Red,citecolor=Green]{hyperref}
\usepackage{color}
\usepackage[font=footnotesize]{caption}
\usepackage{a4wide}
\newtheorem{theo}{Theorem}
\newtheorem{prop}{Proposition}[section]

\newtheorem{lemm}[prop]{Lemma}
\newtheorem{corr}[prop]{Corollary}

\numberwithin{equation}{section}
\newcommand{\mc}{\mathcal}
\newcommand{\rr}{\mathbb{R}}
\newcommand{\nn}{\mathbb{N}}
\newcommand{\cc}{\mathbb{C}}
\newcommand{\hh}{\mathbb{H}}

\newcommand{\eps}{\epsilon}
\newcommand{\pl}{\partial}
\newcommand{\x}{\times}

\newcommand{\til}{\widetilde}

\newcommand{\cjd}{\rangle}
\newcommand{\cjg}{\langle}

\title{Marked boundary rigidity for surfaces}

\author{Colin Guillarmou}
\address{DMA, UMR 8553 CNRS, \'Ecole Normale Sup\'erieure\newline\indent 45 rue d'Ulm, 75230 Paris cedex 05, France}
\email{cguillar@dma.ens.fr}

\author{Marco Mazzucchelli}
\address{UMPA, UMR 5669 CNRS, \'Ecole Normale Sup\'erieure de Lyon\newline\indent  46 all\'ee d'Italie, 69364 Lyon Cedex 07, France}
\email{marco.mazzucchelli@ens-lyon.fr}

\begin{document}

\date{February 9, 2016. \emph{Revised}: July 25, 2016}

\maketitle

\begin{abstract}
We show that, on an oriented compact surface, two sufficiently $C^2$-close  Riemannian metrics with strictly convex boundary, no conjugate points, hyperbolic trapped set for their geodesic flows, and same marked boundary distance, are isometric via a diffeomorphism that fixes the boundary. We also prove that the same conclusion holds on a compact surface for any two negatively curved Riemannian metrics with strictly convex boundary and same marked boundary distance, extending a result of Croke and Otal. 
\end{abstract}

\section{Introduction}

In a smooth oriented compact Riemannian surface $(M,g)$ with strictly convex boundary $\pl M$ and no conjugate points there is a unique geodesic in each homotopy class of curves joining any given pair of points on $\pl M$. 
The \emph{boundary distance function} can be defined as the restriction to $\pl M\x \pl M$ of the Riemannian distance function of $(M,g)$. This function measures the length of the minimizing geodesic joining the given   points on $\pl M$. If $M$ is not simply connected, in general we cannot expect to recover much information on the isometry class of $(M,g)$ from the boundary distance function alone: indeed, all the minimizing geodesics between boundary points may be contained in a rather small neighborhood of $\pl M$. Instead, it is natural to consider the set of lengths of all geodesics joining boundary points, and to study the associated rigidity problem. More precisely, we consider the \emph{lens data} of $g$, which consists of a pair of functions $(\ell_g,S_g)$ defined as follows.
The \emph{scattering map} $S_g$ takes a vector of unit length $(x,v)\in T_{\pl M}M$  pointing inside $M$, and returns the exit tangent vector $S_g(x,v)=(x',v')\in T_{\pl M}M$ obtained by following the geodesic 
$\alpha_{x,v}$ of $(M,g)$ starting at $x$ with tangent vector $v$. The \emph{length function} $\ell_g$ takes an inward-pointing unit vector  $(x,v)\in T_{\pl M}M$ and returns the length of the geodesic $\alpha_{x,v}$; note that $\ell_g(x,v)$ may be infinite at some points $(x,v)$. The lens rigidity problem asks whether the lens data $(\ell_g,S_g)$ determines the metric $g$ up to an isometry that is the identity on $\pl M$. 
In this paper, we also consider another related quantity that we call the \emph{marked boundary distance}, defined as follows.
There is a bundle $\mc{P}_{\pl M} \to \pl M\x \pl M$  whose fiber over $(x,x')\in \pl M\x \pl M$ is the set of homotopy classes of curves joining $x$ and $x'$. The \emph{marked boundary distance} is the function on $\mc{P}_{\pl M}$ given by
\[ (x,x';[\gamma])\mapsto {\rm length}(\alpha(x,x';[\gamma])),\]
where $\alpha(x,x';[\gamma])$ is the unique geodesic with endpoints $x$ and $x'$ 
in the homotopy class given by $[\gamma]$. The question that we address here is 
the rigidity of Riemannian metrics with the same marked boundary distance, much in the spirit of Otal's paper 
\cite{Otal:1990ta} on the marked length spectrum rigidity for closed negatively-curved surfaces. We remark that the marked boundary distance function gives the same information as the lens data on the universal cover of $M$, see Section~\ref{s:hyp_trapped_set}.

In order to describe our results, we first need to recall a few notions.
As usual, we denote by $SM=\{(x,v)\in TM\ |\ g_x(v,v)=1\}$  the unit tangent bundle, and by $\varphi_t$ the geodesic 
flow on $SM$ at time $t$, which is generated by the geodesic vector field $X$.
For all points $(x,v)\in SM$, the geodesic  
 with initial point $x$ and tangent vector $v$ has either infinite length or 
 exits $M$ at some boundary point $x'\in \pl M$. We define the \emph{trapped set} $K\subset SM^\circ$ of $g$
to be the set of points that do not reach the boundary in forward nor in backward time:
\[K := \big\{(x,v)\in SM^\circ\ \big|\ \forall t\in \rr, \, \varphi_t(x,v)\in SM^\circ \big\}.\] 
Due to the strict convexity of $\pl M$, $K$ is compact and invariant by the geodesic flow. In Section~\ref{s:hyp_trapped_set}, we will prove  that if a Riemannian metric $g$ on a compact surface has geodesic flow with hyperbolic trapped set  (see the definition before equation~\eqref{splitting}), strictly convex boundary, and no conjugate points, then the same is true for all smooth Riemannian metrics in a sufficiently small $C^2$-neighborhood of $g$.

Our first theorem is a local marked boundary rigidity result for surfaces with  hyperbolic trapped set and no conjugate points. Its proof is based on the recent work \cite{Guillarmou:2014jk} of the first author.
\begin{theo}\label{Th1}
Let $(M,g_1)$ be a smooth oriented compact Riemannian surface with strictly convex boundary, no conjugate points, and whose trapped set $K$ is hyperbolic. 
There exists $\delta>0$ such that, for each smooth  Riemannian metric $g_2$ on $M$  satisfying 
$\|g_2-g_1\|_{C^2}<\delta$ and having the same marked boundary distance of $g_1$,
there is a smooth diffeomorphism $\phi: M\to M$ with $\phi^*g_2=g_1$ and $\phi|_{\pl M}={\rm Id}$.
\end{theo}

Our second theorem deals with the case where one of the metric has negative curvature, and can be seen as an extension of a celebrated result due to Croke \cite{Croke:1990fi} and Otal \cite{Otal:1990zh}. The method that we employ is essentially the one of \cite{Otal:1990zh}.
\begin{theo}\label{Th2}
Let $M$ be a smooth oriented compact surface with boundary. 
Let $g_1$ be a smooth Riemannian metric on $M$ with negative curvature and strictly convex boundary, and $g_2$ a smooth Riemannian metric on $M$ with strictly convex boundary, no conjugate points and trapped set of zero Liouville measure.  
If $g_1$ and $g_2$ have the same marked boundary distance, 
there is a diffeomorphism $\psi: M\to M$ with $\psi^*g_2=g_1$ and $\psi|_{\pl M}={\rm Id}$.
\end{theo}

We remark that the assumptions on the Riemannian metric $g_2$ in Theorem \ref{Th2} are satisfied if $(M,g_2)$ has negative curvature and strictly convex boundary.

In order to put our theorems into perspective, let us recall a few results on related rigidity questions for surfaces. Riemannian metrics with strictly convex boundary, no conjugate points, and empty trapped set are called \emph{simple metrics}. Their underlying compact surfaces are topological disks. In 1981, Michel conjectured that simple Riemannian metrics on compact surfaces are boundary rigid \cite{Michel:1981pz}, i.e.\ the boundary distance function determines the Riemannian metric within the class of simple metrics up to isometry. Mukhometov \cite{Muhometov:1981jh} proved this conjecture under the extra assumption that the simple metrics considered belong to a given conformal class (see also Croke \cite{Croke:1991le} for a simpler proof). Later, Croke \cite{Croke:1990fi} and Otal \cite{Otal:1990zh} proved Michel's conjecture 
for negatively-curved simple metrics. This was extended to non-positively curved simple metrics by Croke, Fathi, and Feldman \cite{Croke:1992qf}. The conjecture for simple surfaces was finally established by Pestov and Uhlmann \cite{Pestov:2005bu}.
As for non simple metrics on compact surfaces, Croke and Herreros \cite{Croke:2011pd} proved that a negatively-curved cylinder, the flat cylinder, and the flat M\"obius strip are lens rigid. In \cite{Guillarmou:2014jk}, the first author showed that the scattering map $S_g$ determines the compact manifold and the Riemannian metric, within the class considered in Theorem~\ref{Th1}, up to a conformal diffeomorphism.  

\subsection*{Acknowledgements.} C.G. is partially supported by the ANR projects GERASIC (ANR-13-BS01-0007-01) and IPROBLEMS (ANR-13-JS01-0006). M.M. is partially supported by the ANR projects WKBHJ (ANR-12-BS01-0020) and COSPIN (ANR-13-JS01-0008-01).

\section{Manifolds with hyperbolic trapped set and no conjugate points}\label{s:hyp_trapped_set}
 \subsection{Geometric and dynamical preliminaries}
Let $(M,g)$ be a compact Riemannian surface with strictly convex boundary. As before, we denote by $X$ its geodesic vector field on $SM$, and by $\varphi_t$ the corresponding geodesic flow. We also equip $SM$ with a contact form $\lambda$, which is obtained by pulling back the canonical Liouville form on $T^*M$ via the diffeomorphism $(x,v)\mapsto(x,g_x(v,\cdot))$. The Liouville measure $\mu_g$ on $SM$ is defined as the measure associated to the density $|\lambda\wedge d\lambda|$. We consider the pull-back Riemannian metric $\til{g}=\pi^*g$ on the universal cover $\pi: \til{M}\to M$. The fundamental group $\pi_1(M)=\pi_1(M,x_0)$, for some fixed $x_0\in M$, acts by isometries on $(\til{M},\til{g})$. 
The universal cover $\til{M}$ is non-compact and has non-compact boundary $\pl\til{M}$ that is a countable union of connected components and projects to $\pl M$. The unit tangent bundle $S\til{M}$ of $\til{M}$ is also a cover of $SM$ and, by abuse of notation, we still denote by $\pi:S\til{M}\to SM$ the covering map. We will denote by $\til{\varphi}_t$ and $\mu_{\til g}$ the geodesic flow and the Liouville measure on $S\til M$. Notice that $\til{\varphi}_t$ is a lift of $\varphi_t$, i.e.\ $\varphi_t\circ\pi=\pi\circ\til{\varphi}_t$, and if $\gamma.(x,v)$ denotes the natural action of $\pi_1(M)$ on $SM$, we have 
\begin{align*}
\til{\varphi}_t(\gamma\cdot(x,v))=\gamma\cdot(\til{\varphi}_t(x,v)),
\qquad
\forall (x,v)\in S\til{M},\ \gamma\in \pi_1(M).
\end{align*}
We define the incoming ($-$), outgoing ($+$) and tangent ($0$) boundaries of $SM$ 
\begin{align*}
\pl_\mp SM & :=\{ (x,v)\in \pl SM\ |\  \pm \cjg v,\nu\cjd> 0\}, \\ 
\pl_0SM & :=\{ (x,v)\in \pl SM\ |\  \cjg v,\nu\cjd= 0\}.
\end{align*}
where $\nu$ is the inward-pointing unit normal vector field to $\pl M$.
For each $(x,v)\in SM$, we define the time of escape from $M$ by
\begin{equation*}
\ell_g(x,v):=\sup\, \{ t\geq 0\ |\  \varphi_{s}(x,v)\in SM^\circ\ \ \forall s<t \}\subset [0,+\infty].
\end{equation*}
We define $\pl_\pm \til{M}$, $\pl_0S\til{M}$, and the time of escape 
$\ell_{\til{g}}:S\til{M}\to [0,\infty]$ analogously. Notice that $\ell_{\til{g}}=\ell_{g}\circ\pi$. We define the incoming $(-)$ and outgoing $(+)$ tails in $SM$ as  
\[\Gamma_\mp :=\{(x,v)\in SM\ |\  \ell_g(x,\pm v)=\infty\},\]
and the trapped set for the geodesic flow as
\begin{equation}\label{trapped} 
K:=\Gamma_+\cap \Gamma_-.
\end{equation}
Using the strict convexity of $\pl M$, it is straightforward to see that that $\Gamma_\pm$ and $K$ are compact, and $K$ is a subset of $SM^\circ$  invariant by the geodesic flow. 
By the strict convexity of the boundary, the argument in \cite[Section~5.1]{Dyatlov:2014gf} shows that 
\begin{equation}\label{muK}
\mu_g(K)=0 \iff \mu_g(\Gamma_+\cup\Gamma_-)=0.
\end{equation}
We say that  $K$ is \emph{hyperbolic} when, for all $y\in K$, there is a $d\varphi_t$-invariant splitting 
\begin{align}\label{splitting}
T_y(SM)=\rr X(y)\oplus E_u(y)\oplus E_s(y),
\end{align}
where  $E_u$ and $E_s$ are continuous subbundles over $K$ satisfying, for some $C,\alpha>0$ and for all $y\in K$,
\begin{equation}\label{contraction} 
 \begin{split}
 \zeta\in E_s(y) & \qquad\mbox{if and only if}\qquad \|d\varphi_t(y)\zeta\|_g\leq Ce^{-\alpha  t}\|\zeta\|_g,\ \ \forall t\geq 0, \\
 \zeta \in E_u(y) &\qquad\mbox{if and only if}\qquad  \|d\varphi_t(y)\zeta\|_g\leq Ce^{\alpha t}\|\zeta\|_g,\ \  \forall t\leq 0.
 \end{split} 
\end{equation}
The norms in these expressions are given by the Sasaki metric on $SM$, see e.g.\ \cite{Paternain:1999uq}. 
We define the vertical bundle by 
\[V:=\ker(d\pi)\subset T(SM),\] 
where $\pi:SM\to M$ is the base projection. We say  that $x_-,x_+\in M$ are \emph{conjugate points} with respect to the geodesic  flow $\varphi_t$ if there exists $t_0\in\rr$ such that $\varphi_{t_0}(x_-)=x_+$ and  
$(d\varphi_{t_0}(x_-)V(x_-))\cap V(x_+)\not=\{0\}$.
By \cite[Lemma 2.10]{Dyatlov:2014qy}, there is a continuous extension $E_-\subset T_{\Gamma_-}(SM)$ of $E_s$ over $\Gamma_-$ that is invariant by the geodesic flow and satisfies contraction estimates of the form \eqref{contraction}. If  
$(M,g)$ has no conjugate points,  $V\cap E_-=\{0\}$, see \cite[Section 2.2]{Guillarmou:2014jk}. There is an analogous extension $E_+\subset T_{\Gamma_+}(SM)$ of $E_u$ over $\Gamma_+$.

\begin{prop}\label{stability}
Let $(M,g)$ be a compact Riemannian surface with strictly convex boundary, no conjugate points, and with hyperbolic trapped set for its geodesic flow. There exists $\delta>0$ such that, for each Riemannian metric $g'$ satisfying $\|g'-g\|_{C^2}<\delta$, the Riemannian surface  $(M,g')$ has strictly convex boundary, no conjugate points, and hyperbolic trapped set for its geodesic flow.
\end{prop}

\begin{proof}
We denote by $\varphi_t$ the geodesic flow of the Riemannian metric $g$, by $\Gamma_{\mp}\subset SM$ its incoming and outgoing tails, and by $K\subset SM^\circ$ its trapped set as usual. Consider the family of open sets
\begin{align}\label{e:fundamental_system}
U_{\tau,\pm}
:=
\big\{y\in SM\ \big|\  \varphi_{\mp t}(y)\in SM^\circ\ \ \forall t\in[0,\tau]
\big\},
\end{align}
for $\tau>0$. This family forms a fundamental system of open neighborhoods of $\Gamma_\pm$, that is, for every open neighborhood $U\subset SM$ of $\Gamma_\pm$ there exists $\tau>0$ large enough such that $U_{\tau,\pm}\subset U$. Analogously, the intersections $U_{\tau}:=U_{\tau,+}\cap U_{\tau, -}$ form a fundamental system of open neighborhoods of the trapped set $K$. We refer the reader to~\cite{Dyatlov:2014qy} for a proof of these facts.

Consider the hyperbolic splitting~\eqref{splitting} over $K$, and the extensions $E_\mp\subset T_{\Gamma_\mp}(SM)$ of the stable and unstable bundles. We further extend $E_\mp$ continuously over a neighborhood $U_{\tau_0,\mp}$ of $\Gamma_\pm$, and define the stable and unstable cones as
\begin{align*}
C_{\mp,\rho}(y)
&:=
\big\{
\xi+\eta\ \big|\ \xi\in E_\mp(y),\ \eta\in E_\pm(y)\oplus \rr X(y),\ \|\eta\|_{g}\leq\rho\|\xi\|_{g}
\big\},
\end{align*}
where $y\in U_{\tau_0,\mp}$ and $\rho>0$.  
We claim that, for all large enough real numbers $\tau_0>0$ and $\alpha_0>0$, we have 
\begin{align}\label{e:vertical_into_cone}
d\varphi_{\pm t}(y)V(y) \Subset C_{\pm ,\alpha_0}(\varphi_{\pm t}(y)),
\qquad
\forall 
y\in U_{2\tau_0,\mp},\ 
\forall t\in [\tau_0,2\tau_0]. 
\end{align}
Indeed, if this is not the case, by a compactness argument, for all sequences $\tau_n\to +\infty$ there exist $y^\mp_n\in U_{2\tau_n,\mp}$ and $t_n>\tau_n$ such that $d\varphi_{\pm t_n}(y^\mp_n)V(y^\mp_n)\cap E_\mp(\varphi_{\pm t_n}(y^\mp_n))\not=0$. Since the Riemannian manifold $(M,g)$ has no conjugate points, the vertical bundle $V$ is contained in a conic 
neighborhood $C_{\pm,\rho}$, for some $\rho>0$, whose closure does not contain $E_\mp$. Therefore, we can apply Lemma~2.11 of \cite{Dyatlov:2014qy} to deduce that the distance of $d\varphi_{\pm t_n}(y^\mp_n)V(y^\mp_n)$ to $E_+$ 
in the Grassmannian bundle of $SM$ tends to $0$ as $n\to \infty$, which gives a contradiction.

We fix $\tau_0>0$ large enough so that~\eqref{e:vertical_into_cone} holds and, for some $\rho_0>0$ small enough, we have
\begin{align*}
 C_{+,\rho_0}(y)\cap C_{-,\rho_0}(y) = \{0\},\qquad\forall y\in U_{\tau_0}.
\end{align*}
By the hyperbolicity of the flow on $K$, for each pair of positive numbers $\alpha\geq\rho>0$ there exists $\overline \tau=\overline\tau(\alpha,\rho)>2\tau_0$ large enough such that, for all  $t\geq\overline\tau$ and  $y\in U_{\tau_0} \cap \varphi_{\mp t}^{-1}(U_{\tau_0})$,
\begin{equation}
\label{e:expansion_g_1}
\begin{split}
& d\varphi_{\mp t}(y)C_{\mp,\alpha}(y) \Subset C_{\mp,\rho}(\varphi_{\mp t}(y)),\\
& \|d\varphi_{\mp t}(y)\zeta\|_{g}  \geq 4\|\zeta\|_{g},\quad  \forall\zeta\in C_{\mp,2\rho}(y).
\end{split} 
\end{equation}

There exists $\delta=\delta(\overline{\tau})>0$ such that, for any Riemannian metric $g'$ on $M$ satisfying $\|g'-g\|_{C^2}<\delta$, the boundary of $(M,g')$ is strictly convex and no pair of points of $(M,g')$ are conjugate along a geodesic of length less than or equal to $4\overline{\tau}$.  Let $\psi:SM\to S'M$ be the diffeomorphism given by $\psi(x,v)=(x,v/\|v\|_{g'})$, where $S'M$ is the unit tangent bundle of $M$ with respect to  $g'$. Since $\psi$ preserves the fibers of the unit tangent bundles, its differential maps the vertical subbundle $V\subset T(SM)$ to the vertical subbundle of $T(S'M)$. From now on, we identify $SM$ with $S'M$ by means of $\psi$, and thus we see the geodesic flow $\varphi'_t$ of $g'$ as a flow on $SM$. We define the open sets $U_{\tau,\pm}'$ analogously to~\eqref{e:fundamental_system}, that is,
\begin{align*}
U_{\tau,\pm}'
:=
\big\{y\in SM\ \big|\  \varphi_{\mp t}'(y)\in SM^\circ\ \ \forall t\in[0,\tau]
\big\},
\end{align*}
and $U_\tau':=U_{\tau,-}'\cap U_{\tau,+}'$. Notice that $\varphi'_t$ is $\mathcal{O}(\delta)$-close to $\varphi_t$ in the $C^1$-topology for all $t\in[0,2\tau_0]$. In particular, up to further reducing $\delta>0$, we have $U_{2\tau_0,\mp}'\subset U_{\tau_0,\mp}$, and the following holds: by~\eqref{e:expansion_g_1}, for all  $t\in[\overline\tau,2\overline{\tau}]$ and  $y\in U_{2\tau_0}' \cap \varphi_{\pm t}'(U_{2\tau_0}')$
we have
\begin{equation}
\label{e:expansion_g_2}
\begin{split}
& d\varphi_{\mp t}'(y)C_{\mp,\alpha_0}(y) \Subset C_{\mp,\rho_0}(\varphi_{\mp t}'(y)),\\
& \|d\varphi_{\mp t}'(y)\zeta\|_{g'}  \geq 2\|\zeta\|_{g'},\quad  \forall\zeta\in C_{\mp,2\rho_0}(y);
\end{split} 
\end{equation}
in addition, by \eqref{e:vertical_into_cone}, we have
\begin{align}\label{e:vertical_into_cone_1}
d\varphi_{\pm t}'(y)V(y) \Subset C_{+,\alpha_0}(\varphi_{\pm t}'(y)),
\qquad
\forall 
y\in U_{2\tau_0,\mp}',\ 
t\in[\tau_0,2\tau_0]. 
\end{align}
Equations~\eqref{e:expansion_g_2} and~\eqref{e:vertical_into_cone_1} imply that
\begin{align}
\label{e:vertical_into_cone_2}
d\varphi_{\pm t}'(y)V(y) \Subset C_{\pm,\rho_0}(\varphi_{\pm t}'(y)),
\qquad
\forall 
y\in U_{2\tau_0,\mp}'\cap \varphi_{\mp t}'(U_{2\tau_0}'),\ 
t\geq2\overline\tau,
\end{align}

By \cite[Proposition 17.4.4]{Katok:1995xi}, the set
\begin{align*}
K':=\bigcap_{t\in\rr}\varphi'_t(U_{2\tau_0}'),
\end{align*}
which is the trapped set for the geodesic flow $\varphi_t'$, is hyperbolic. We claim that the Riemannian metric $g'$ has no conjugate points. Indeed, assume that two points $\gamma(0)$ and $\gamma(4\tau)$ are conjugate along a unit-speed geodesic $\gamma:[0,4\tau]\to M$ of $(M,g')$. By the property of the Riemannian metric $g'$, we must have $\tau>\overline{\tau}$. We set
\[y_-  :=(\gamma(0),\dot{\gamma}(0)),\qquad y_+  :=(\gamma(4\tau),\dot{\gamma}(4\tau)).\]
Notice that 
$y_-\in U_{2\tau_0,-}'$ and $y_+\in U_{2\tau_0,+}'$. Therefore, by \eqref{e:vertical_into_cone_2}, we have  
\begin{align*}
d\varphi_{\pm2\tau}'(y_\mp)V(y_\mp) &\Subset C_{\pm,\rho_0}(\varphi_{\pm2\tau}'(y_\mp)).
\end{align*}
Since $C_{+,\rho_0}\cap C_{-,\rho_0}=\{0\}$ over $U'_{2\tau_0}$, we deduce that $d\varphi_{4\tau}'(y_-)V(y_-)\cap V(y_+)=\{0\}$, which contradicts the facts that $\gamma(0)$ and $\gamma(4\tau)$ are conjugate along $\gamma$.
\end{proof}

We now recall Santal\'o's formula corresponding to the disintegration of the Liouville measure $\mu_g$ on geodesics. Let $h:=g|_{T\pl M}$ be the metric induced on the boundary. 
There is a natural measure on $\pl SM$, defined by 
\begin{equation}\label{dmunu}
d\mu_{g,\nu}(x,v):= |g_x(v,\nu_x)|\, d{\rm vol}_{h}(x)\,dS_x(v).
\end{equation}
where $\nu$ is the inward-pointing unit normal vector field to $\pl M$, $d{\rm vol}_{h}$ is the Riemannian volume form on $(\partial M,h)$, and $dS_x$ is the volume measure  
on the fiber $S_xM$ induced by $g$.  
Assume that $\mu_g(\Gamma_-\cup \Gamma_+)=0$, or equivalently $\mu_g(K)=0$ according to \eqref{muK}. Santal\'o's formula reads 
\begin{equation}\label{santalo} 
\int_{SM}fd\mu_g = \int_{\pl_-SM\setminus \Gamma_-}\int_{0}^{\ell_g(x,v)}f(\varphi_t(x,v)) dt \,d\mu_{g,\nu}(x,v),\qquad \forall f\in L^1(SM).
\end{equation}
In particular, for $f\equiv 1$, Santal\'o's formula computes the Riemannian volume of $SM$ as
\begin{align}\label{Volume}
\mathrm{Vol}_{g}(SM) = \int_{\pl_-SM\setminus \Gamma_-} \ell_g(x,v) \,d\mu_{g,\nu}(x,v).
\end{align}
We define the incoming and outgoing tails in $S\til M$ as before by
\[\til{\Gamma}_\mp :=\{(x,v)\in SM\ |\  \ell_{\til{g}}(x,\pm v)=\infty\}.\]
Notice that $\pi(\til\Gamma_\mp)=\Gamma_\mp$. If $K$ is hyperbolic,  
the sets $\Gamma_\pm$  have measure zero with respect to $\mu_g$, see \cite[Section 2.4]{Guillarmou:2014jk}; in particular, $\til{\Gamma}_\pm$ have measure zero with respect to $\mu_{\til g}$ as well. We denote by $d\mu_{\til g,\nu}$ the lift of $d\mu_{g,\nu}$ to $\pl S\til{M}$. Since $\til{\Gamma}_\pm$ have also measure zero with respect to $\mu_{\til g}$, Santal\'o's formula on $S\til{M}$ reads  
\begin{equation}\label{santalo'} 
\int_{S\til{M}}fd \mu_{\til g} = \int_{\pl_-S\til{M}\setminus \til{\Gamma}_-}\int_{0}^{\ell_{\til{g}}(x,v)}
f(\til{\varphi}_t(x,v)) dt \,d\mu_{\til g,\nu}(x,v),\qquad
\forall f\in L^1(S\til{M}).
\end{equation}

\subsection{Scattering map, lens data, and boundary distance functions}
We define the \emph{scattering map} of $(M,g)$ as follows
\begin{equation}\label{scattering} 
S_g :  \pl_-SM\setminus \Gamma_- \to \pl_+SM\setminus \Gamma_+, \quad S_g(x,v):=\varphi_{\ell_g(x,v)}(x,v).
\end{equation}
The \emph{lens data} of $(M,g)$ is the pair of functions 
$(S_g,\ell_g|_{\pl_-SM})$, which amounts to knowing 
the length of all geodesics joining boundary points and their tangent vectors at the boundary. The analogous definitions hold for the universal cover $(\til{M},\til{g})$. In particular, the scattering map
\begin{equation}\label{scattering'} 
S_{\til{g}} :  \pl_-S\til{M}\setminus \til{\Gamma}_- \to \pl_+S\til{M}\setminus \til{\Gamma}_+, \quad S_{\til g}(x,v):=\til{\varphi}_{\ell_{\til{g}}(x,v)}(x,v)
\end{equation}
satisfies $S_g\circ\pi=\pi\circ S_{\til{g}}$ and is $\pi_1(M)$-equivariant, i.e.
\begin{equation}\label{equivaraince}
S_{\til{g}}(\gamma\cdot (x,v))=\gamma\cdot(S_{\til{g}}(x,v)),\qquad \forall (x,v)\in \pl_-S\til{M}\setminus \til{\Gamma}_-,\ \gamma\in \pi_1(M).
\end{equation}
The lens data of $(\til{M},\til{g})$ is the pair $(S_{\til{g}},\ell_{\til{g}}|_{\pl_-S\til{M}})$, where 
$\ell_{\til{g}}=\ell_g\circ\pi$.

The following result is well known to the experts. 
\begin{lemm}\label{geodesics}
Let $(M,g)$ be a smooth manifold with smooth, strictly convex boundary and no conjugate points.  Let $\alpha$ be a geodesic with endpoints $x,x'\in M$.
If $\gamma$ is any other smooth curve in $M$ with endpoints $x,x'$ that is homotopic to 
$\alpha$ with a homotopy fixing the endpoints, then the length of $\gamma$  is larger than the length of $\alpha$.
\end{lemm}
\begin{proof}
Let $\Pi:=\{\gamma\in W^{1,2}([0,1]; M)\ |\ \gamma(0)=x, \gamma(1)=x'\}$, where $W^{1,2}([0,1]; M)$ is the 
Hilbert manifold of absolutely continuous maps $\gamma: [0,1]\to M$ whose energy 
\[E(\gamma):=\int_{0}^1 g_{\gamma(t)}(\dot{\gamma}(t),\dot{\gamma}(t))\,dt\] is finite. We also consider the finite dimensional reduction of this space, that is, the space $\Pi_k$ of all curves $\gamma\in\Pi$ such that, for all $i=0,...,k-1$, the curve $\gamma|_{[i/k,(i+1)/k]}$ is a geodesic of length smaller than $\mathrm{injrad}(M,g)$. For all $a>0$ there exists $k\in\nn$ large enough so that the inclusion 
\begin{equation}\label{homot_equiv}
 \Pi_k\cap E^{-1}(-\infty,a)\hookrightarrow E^{-1}(-\infty,a)
\end{equation}
is a homotopy equivalence. We refer the reader to Milnor \cite[Section~16]{Milnor:1963wm} for this and other properties of the space $\Pi_k$.
The connected components of $\Pi$ are homotopy classes of curves in $M$ with fixed endpoints $x$ and $x'$. We choose a connected component $\Pi'\subset\Pi$ containing a geodesic $\alpha$. 
First, we notice that a curve $\alpha'$ which minimizes  the energy $E$ in $\Pi'$ is in fact a smooth geodesic that intersects $\pl M$ only at its endpoints $x,x'$. This follows from the strict convexity of the boundary  (see for instance the arguments given in \cite[Section 1]{Otal:1990zh}). In particular, the minimizers of $E|_{\Pi'}$ lie in the interior of $\Pi'$. The critical points of the energy $E|_{\Pi'}$ are the geodesics in the connected component $\Pi'$. Since $g$ has no conjugate points, the Morse index Theorem implies that the Hessian of $E$ at the critical points is positive definite (see for instance Milnor \cite[Section 15]{Milnor:1963wm}). The same is true for the restriction of $E$ to $\Pi_k$. We will use a mountain pass argument to prove that $\alpha$ is the unique minimizer in $\Pi'$. 
If $\alpha'$ is another local minimizer $E|_{\Pi'}$, consider a continuous path $A:[0,1]\to\Pi$ joining $\alpha$ and $\alpha'$. We set $a>\max\{E(A(r))\ |\ r\in[0,1]\}$, and choose $k\in\nn$ large enough so that the inclusion \eqref{homot_equiv} is a homotopy equivalence. In particular $\alpha$ and $\alpha'$ belong to the same connected component of $\Pi_k\cap E^{-1}(-\infty,a)$. We denote by $E_k$ the restriction of $E$ to $\Pi_k\cap E^{-1}(-\infty,a)$, and by $\mathcal{P}$ the space of continuous paths $B:[0,1]\to\Pi_k\cap E^{-1}(-\infty,a)$ with $B(0)=\alpha$ and $B(1)=\alpha'$. Since the boundary of $M$ is strictly convex, a well known computation (see \cite[page~252]{Gromoll:1975jb}) shows that the flow of $-\nabla E_k$ is positively complete. Therefore, the minimax value 
\begin{align*}
b:= \inf_{B\in\mathcal{P}} \max_{s\in[0,1]} E_k(B(s))
\end{align*}
is a critical value of $E_k$. By a theorem of Hofer \cite{Hofer:1985kr}, there exists a critical point $\alpha''$ with $E_k(\alpha'')=b$ that is not a local minimum of $E_k$. By the Morse index Theorem, $\alpha''$ must contain a pair of conjugate points. This gives a contradiction.
\end{proof}

We denote by $\mc{P} \to M\x M$ the bundle whose fiber $\mc{P}_{(x,x')}$ over $(x,x')\in M\times M$ is the set homotopy classes of curves $\gamma:[0,1]\to M$ such that $\gamma(0)=x$ and $\gamma(1)=x'$. The fundamental group $\pi_1(M,x)$ acts freely and transitively on every fiber by concatenation: given a homotopy class $[\gamma]\in\mc{P}_{(x,x')}$ and an element $[\zeta]\in\pi_1(M,x)$, the action is given by $[\zeta]\cdot[\gamma]=[\zeta*\gamma]$. Therefore, the $\mc{P}_{(x,x')}$ is in one-to-one correspondence with $\pi_1(M,x)$. We  define the \emph{marked distance} function
\[ d_g : \mc{P} \to [0,\infty),\quad d_g(x,x';[\gamma])={\rm length}(\alpha_{g}(x,x';[\gamma])),\]
where $\alpha_{g}(x,x';[\gamma])$ is the unique geodesic with endpoints $x$ and $x'$ 
in the homotopy class $[\gamma]\in\mc{P}_{(x,x')}$. We denote by $\mc{P}_{\pl M}$ the restriction of the bundle $\mc{P}$ to $\pl M\x \pl M$.

A consequence of Lemma \ref{geodesics} is that for each pair of points $x,x'$ on $\til{M}$, there is a unique geodesic $\alpha_{\til g}(x,x')$
joining $x$ and  $x'$ in $\til{M}$. We define the distance function
\begin{equation}\label{distancecover}
d_{\til{g}}: \til{M}\x \til{M} \to [0,\infty) ,\qquad d_{\til{g}}(x,x')={\rm length}(\alpha_{\til{g}}(x,x')).
\end{equation}
Notice that
\[ d_{\til{g}}(x,x')=d_{g}(\pi(x),\pi(x');[\pi\circ\alpha_{\til{g}}(x,x')]).\]
We call \emph{marked boundary distance} the restriction $d_g|_{\mc{P}_{\pl M}}$, 
the knowledge of which is equivalent to the knowledge of $d_{\til{g}}|_{\pl\til{M}\x\pl\til{M}}$.

The next two lemmas are well known for simple metrics. Since our assumptions are weaker, we provide the proofs for the reader's convenience.

\begin{lemm}\label{bdrydistgivemet}
Let $M$ be a compact manifold with smooth boundary, and $g_1,g_2$ two Riemannian metrics on $M$ with the same boundary distance function (i.e.\ $d_{g_1}=d_{g_2}$ on $\pl M\x \pl M$), and such that the boundary $\partial M$ is convex with respect to both metrics. Then there exists a diffeomorphism $\psi:M\to M$ that is the identity on $\partial M$ and such that $g_2$ and $\psi^*g_1$ coincide at all points of $\partial M$.
\end{lemm}

\begin{proof}
Let $\iota:\partial M\hookrightarrow M$ be the inclusion map of the boundary inside our compact manifold. We set $h_1:=\iota^*g_1$ and $h_2:=\iota^*g_2$. We claim that these two Riemannian metrics coincide. Indeed, since the boundary $\partial M$ is convex for both $g_1$ and $g_2$, their Riemannian distances are smooth in a neighborhood of the diagonal submanifold of $M\times M$. For any given $(x,v)\in T\partial M$, if $\gamma:[0,1]\to\partial M$ is any smooth curve such that $\gamma(0)=x$ and $\dot\gamma(0)=v$, we have 
\begin{align*}
h_1(v,v)^{1/2}
= \lim_{t\to0} \frac{d_{g_1}(\gamma(0),\gamma(t))}{t}
= \lim_{t\to0} \frac{d_{g_2}(\gamma(0),\gamma(t))}{t}
=
h_2(v,v)^{1/2}.
\end{align*}

Now, for each $x\in\partial M$, we denote by $\nu_1(x),\nu_2(x)\in T_x M$ the inward-pointing unit normal tangent vectors with respect to $g_1$ and $g_2$ respectively. For $i=1,2$, we introduce the map
\begin{align*}
\phi_i:\partial M\times[0,\epsilon_i)\to M,\qquad \phi_i(x,t)=\exp_x(t\nu_i(x)), 
\end{align*}
which, up to choosing $\epsilon_i>0$ small enough, is a well defined diffeomorphism onto a collar neighborhood of the boundary $\partial M$. For $\delta_0>0$ small enough, the map 
\begin{align*}
\theta:\partial M\times[0,\delta_0)\to\partial M\times[0,\epsilon_1),\qquad
\theta(x,t)=\phi_1^{-1}\circ\phi_2(x,t)
\end{align*}
is a well defined diffeomorphism onto its image. We write this diffeomorphism as $\theta(x,t)=(\theta_1(x,t),\theta_2(x,t))$, where $\theta_1(x,t)\in\partial M$ and $\theta_2(x,t)\in[0,\epsilon_1)$. Up to reducing $\delta_0>0$, we have that $\tfrac{\partial}{\partial t}\theta_2$ is everywhere positive, and for all $t\in[0,\delta_0)$ the map $x\mapsto\theta_1(x,t)$ is a diffeomorphism onto its image. For $\delta_1\in[0,\delta_0/2)$ sufficiently small, we can find a smooth function $\theta_2':\partial M\times[0,\delta_0)\to[0,\delta_0)$ such that $\tfrac{\partial}{\partial t}\theta_2'$ is everywhere positive, $\theta_2'(\cdot,t)=\theta_2(\cdot,t)$ for all $t\in[0,\delta_1]$, and $\theta_2'(\cdot,t)\equiv t$ for all $t\in[\delta_0/2,\delta_0)$. Let $\delta_2\in(0,\delta_1/4)$ be a constant that we will fix sufficiently small in a moment. We choose a smooth function $\chi:[0,\delta_0)\to[0,\delta_0)$  such that $\chi(t)=t$ for all $t\in[0,\delta_2]$, $\chi(t)=0$ for all $t\in[3\delta_2,\delta_0]$, and $|\dot\chi(t)|\leq 1$ for all $t\in[0,\delta_0)$. We define the map
\begin{align*}
\theta_1':\partial M\times[0,\delta_0)\to \partial M,\qquad \theta_1'(x,t)=\theta_1(x,\chi(t)).
\end{align*}
Notice that, for all $t\in[0,\delta_0)$, the map $x\mapsto\theta_1'(x,t)$ is a diffeomorphism onto its image, and such diffeomorphism is the identity if $t=0$ or if $t\in[3\delta_2,\delta_0]$. We define
\begin{align*}
\theta':\partial M\times[0,\delta_0)\to \partial M\times[0,\delta_0),\qquad \theta'(x,t)=(\theta_1'(x,t),\theta_2'(x,t)).
\end{align*}
This map is clearly bijective and, up to choosing $\delta_2>0$ sufficiently small, its differential is everywhere bijective. Therefore, $\theta'$ is a diffeomorphism that coincides with $\theta$ on $\partial M\times[0,\delta_2)$, and with the identity on $\partial M\times(\delta_0/2,\delta_0)$. Finally, we set
\begin{align*}
\phi_1':\partial M\times[0,\delta_0)\to M,\qquad \phi_1'(x,t)=\phi_1\circ\theta'(x,t). 
\end{align*}
Notice that $\phi_1'$ is a diffeomorphism onto its image, coincides with $\phi_2$ on $\partial M\times[0,\delta_2)$, and coincides with $\phi_1$ on $\partial M\times(\delta_0/2,\delta_0)$. We set $U:=\phi_1'(\partial M\times[0,\delta))$ and $\psi:=\phi_1\circ (\phi_1')^{-1}:U\to M$.
Since $\psi$ is equal to the identity outside a neighborhood of $\partial M$, we can extend it to a diffeomorphism 
$\psi:M\to M$
by setting $\psi(x)=x$ for all $x\not\in U$. 

We claim that $\psi^*g_1$ coincides with $g_2$ on all points $x\in\partial M$. Indeed, $\psi$ fixes the boundary $\partial M$, and therefore 
\begin{align}\label{iota_psi_g2}
\iota^*\psi^*g_1=\iota^*g_1=\iota^*g_2.
\end{align}
Moreover $\psi$ coincides with $\phi_1\circ \phi_2^{-1}$ on a neighborhood of $\partial M$. Fix a point $x\in\partial M$, and consider the unit-speed $g_2$-geodesic $\gamma_2:[0,\epsilon)\to U$ such that $\gamma_2(0)=x$ and $\dot\gamma_2(0)=\nu_2(x)$. We set $\gamma_1:=\psi\circ\gamma_2$. Notice that, if $\epsilon\in(0,\delta')$, we have $\gamma_1(t)=\phi_1(x,t)$ and $\gamma_2(t)=\phi_2(x,t)$. Therefore the curve $\gamma_1$ is the unit-speed $g_1$-geodesic such that $\gamma_1(0)=x$ and $\dot\gamma_1(0)=\nu_1(x)$, and in particular $d\psi(x)\nu_2(x)=\nu_1(x)$. This implies that, for all $v\in T_x(\partial M)$, we have
\begin{align*}
\psi^*g_1(\nu_2(x),v)=g_1(d\psi(x)\nu_2(x),d\psi(x)v)=g_1(\nu_1(x),v)=0,
\end{align*}
which, together with \eqref{iota_psi_g2}, completes the proof.
\end{proof}

\begin{lemm}\label{lensvsdistance}
Let $M$ be a compact manifold with boundary, and $g_1,g_2$ two Riemannian metrics on $M$ with no conjugate points and such that the boundary $\partial M$ is strictly convex with respect to both metrics. Assume that $g_1$ and $g_2$ coincide at all points of $\pl M$. Then, they have the same marked boundary distance function, i.e.\ $d_{\til{g}_1}=d_{\til{g}_2}$ on $\pl \til{M}\x \pl\til{M}$, if and only if their lens data on $\til{M}$ 
agree, i.e.\ $(S_{\til{g}_1},\ell_{\til{g}_1}|_{\pl_-S\til{M}})=(S_{\til{g}_2},\ell_{\til{g}_2}|_{\pl_-S\til{M}})$.
\end{lemm}
\begin{proof}
Since $g_1$ and $g_2$ are metrics without conjugate points and the boundary $\partial M$ is strictly convex for both of them, any pair of distinct points $x,x'\in \pl\til{M}$ is joined by a unique $\til{g}_1$-geodesic $\alpha_{\til{g}_1}(x,x')$ and by a unique $\til{g}_2$-geodesic $\alpha_{\til{g}_2}(x,x')$. For all distinct points $x,x'\in \pl\til{M}$ we have $d_{\til{g}_1}(x,x')=\ell_{\til{g}_1}(x,v)$, where 
$v$ is the unit tangent vector to $\alpha_{\til{g}_1}(x,x')$ at $x$. If $g_1$ and $g_2$ have  the same lens data on $\til{M}$, we have $S_{\til{g}_2}(x,v)\in S_{x'}\til{M}$ and therefore \[d_{\til{g}_2}(x,x')=\ell_{\til{g}_2}(x,v)=\ell_{\til{g}_1}(x,v)=d_{\til{g}_1}(x,x').\]
Conversely, assume that $d_{\til{g}_1}=d_{\til{g}_2}$ on $\pl\til{M}$. All we need to show is that $\til{g}_1$ and $\til{g}_2$ have the same  scattering maps.  Consider a point $(x,v)\in \pl SM$ that does not belong to the incoming tail of $\til{g}_1$, and set $(x',v_1):=S_{\til{g}_1}(x,v)$. Let $v_2$ be the unit tangent vector to $\alpha_{\til{g}_2}(x,x')$ at $x'$. For $i=1,2$, we define the smooth functions $\zeta_i:=d_{\til{g}_i}(x,\cdot)$ and $\beta_i:=\zeta_i|_{\pl \til{M}}$. We denote by $\nabla^{\til{g}_i}$ the gradient operator with respect to the Riemannian metric $\til{g_i}$, and, by abuse of notation, also the gradient operator with respect to the Riemannian metric induced by $\til{g_i}$ on $\partial\til{M}$. By our assumptions, $\beta_1=\beta_2$ and therefore $\nabla^{\til{g}_1}\beta_1=\nabla^{\til{g}_2}\beta_2$. By Gauss Lemma, $\nabla^{\til{g}_i}\zeta_i(x')=v_i$. Since $\nabla^{\til{g}_i}\beta_i$ is the $\til g_i$-orthogonal projection of $\nabla^{\til{g}_i}\zeta_i$ onto
$T(\pl\til{M})$, and since $v_1,v_2\in T_{x'}\til M$ are outward-pointing unit tangent vectors, we conclude that $v_1=\nabla^{\til{g}_1}\zeta_1(x')=\nabla^{\til{g}_2}\zeta_2(x')=v_2$. Switching the roles of $x$ and $x'$, the same argument shows that the tangent vector to $\alpha_{\til{g}_2}(x,x')$ at $x$ is $v$, and therefore $S_{\til{g}_1}(x,v)=S_{\til{g}_2}(x,v)$. 
\end{proof}

By applying an argument due to Croke \cite[Theorem C]{Croke:1991le} we get the following result.
\begin{lemm}\label{croke}
Let $M$ be a compact surface with smooth boundary, and $g_1,g_2$ two Riemannian metrics on $M$ without conjugate points, such that the boundary $\partial M$ is strictly convex for each one of them, and $g_2=e^{2\omega}g_1$ for some $\omega\in C^\infty(M)$ vanishing at $\pl M$. Assume that the trapped sets of the geodesic flows of $g_1$ and $g_2$ have zero measure for $\mu_{g_1}$ and $\mu_{g_2}$ respectively. If $g_1$ and $g_2$ have the same marked boundary distance function, then $g_1=g_2$.
\end{lemm}  

\begin{proof}
Notice that $g_1$ and $g_2$ coincides on $\partial M$, and they both define the same inward-pointing unit vector field $\nu$ and the same measure $\mu_{g_1,\nu}=\mu_{g_2,\nu}$ on $\partial M$. We denote by $\varphi^i_t$ the geodesic flow of the metric $g_i$. 
Since $d_{\til g_1}=d_{\til g_2}$, Lemma \ref{lensvsdistance} implies that $(S_{\til{g}_1},\ell_{\til{g}_1}|_{\pl_-S\til{M}})=(S_{\til{g}_2},\ell_{\til{g}_2}|_{\pl_-S\til{M}})$. In particular, $S_{g_1}=S_{g_2}$ and $\ell_{g_1}|_{\pl_-SM}=\ell_{g_2}|_{\pl_-SM}$. We denote by $\Gamma^i_{\pm}$ the incoming/outgoing tails of the geodesic flow of $g_i$. Notice that $\Gamma^1_{\pm}\cap\partial_-SM=\Gamma^2_{\pm}\cap\partial_-SM$.  By~\eqref{muK}, these tails have measure zero with respect to the corresponding Liouville measures. Therefore, Santal\'o's formula implies
\begin{equation}
\label{santalo_same_volume}
\begin{split}
{\rm Vol}_{g_1}(M) &= \frac{1}{{\rm Vol}(S^{n-1})} \int_{\partial_-SM\setminus\Gamma_-^1} \ell_{g_1}(x,v) d\mu_{g_1,\nu}(x,v) \\
&= \frac{1}{{\rm Vol}(S^{n-1})} \int_{\partial_-SM\setminus\Gamma_-^2} \ell_{g_2}(x,v) d\mu_{g_2,\nu}(x,v) \\
& ={\rm Vol}_{g_2}(M)
\end{split}
\end{equation}
and
\begin{equation}
\label{santalo_g1_g1}
\begin{split}
{\rm Vol}(S^{n-1})\int_Me^{\omega(x)}d{\rm vol}_{g_1}(x)
&=
\int_{SM} e^{\omega(x)} d\mu_{g_1}(x,v)\\
&=
\int_{SM} |v|_{g_2} d\mu_{g_1}(x,v)\\
&=
\int_{\pl_-SM\setminus\Gamma_-^1}\int_{0}^{\ell_{g_1}(x,v)}
 |\varphi^1_t(x,v)|_{g_2}dt\,  d\mu_{g_1,\nu}(x,v). 
\end{split}
\end{equation}
Consider $(x,v)\in \pl_-SM\setminus\Gamma_-^1$ and, for $i=1,2$, consider the geodesic
\[
\gamma_i:[0,\ell_{g_i}(x,v)]\to M,\qquad (\gamma_i(t),\dot\gamma_i(t))=\varphi_t^i(x,v).
\] 
Since $S_{\til{g}_1}=S_{\til{g}_2}$, we have $\gamma_1(\ell_{g_1}(x,v))=\gamma_2(\ell_{g_2}(x,v))$, and the geodesics $\gamma_1$ and $\gamma_2$ are homotopic via a homotopy that fixes the endpoints. Therefore, by Lemma~\ref{geodesics}, we infer
\begin{align*}
\int_{0}^{\ell_{g_1}(x,v)}|\varphi^1_t(x,v)|_{g_2}dt 
= \mathrm{length}_{g_2}(\gamma_1)
\geq
\mathrm{length}_{g_2}(\gamma_2)
=\ell_{g_2}(x,v)
=\ell_{g_1}(x,v).
\end{align*}
This, together with \eqref{santalo_same_volume} and \eqref{santalo_g1_g1}, implies
\begin{equation}\label{ineq1}
\begin{split}
{\rm Vol}(S^{n-1})\int_Me^{\omega(x)}d{\rm vol}_{g_1}(x)
&\geq
\int_{\pl_-SM\setminus\Gamma_-^1}
\ell_{g_1}(x,v)\,d\mu_{g_1,\nu}(x,v)\\
&= {\rm Vol}(S^{n-1})\,{\rm Vol}_{g_1}(M)\\
&= {\rm Vol}(S^{n-1})\,{\rm Vol}_{g_2}(M). 
\end{split}
\end{equation}
But then the H\"older inequality 
\begin{align*}
{\rm Vol}_{g_2}(M)^{1/2}\,{\rm Vol}_{g_1}(M)^{1/2}\geq \int_{M}e^{\omega(x)}d{\rm vol}_{g_1}(x) 
\end{align*}
is satisfied as an equality, which implies $\omega\equiv 0$.
\end{proof}

\begin{proof}[Proof of Theorem~\ref{Th1}]
Let $g_1$ be a Riemannian metric as in the statement. By Proposition~\ref{stability}, there exists $\delta>0$ such that, if a Riemannian metric $g_2$ on $M$ satisfies $\|g_2-g_1\|_{C^2}<\delta$, then $g_2$ satisfies the same properties as $g_1$ in the statement of Theorem~\ref{Th1}. Assume further that $g_1$ and $g_2$ have the same marked boundary distance. In particular, these metrics have the same boundary distance, and therefore Lemmas \ref{bdrydistgivemet} and \ref{lensvsdistance} imply that they coincide on all points of $\partial M$ and that their scattering maps are the same. By \cite[Theorem~3]{Guillarmou:2014jk}, there exists a diffeomorphism $\phi:M\to M$ which is the identity on $\pl M$ and such that 
$\phi^*g_2=e^{2\eta}g_1$ for some smooth function $\eta\in C^\infty(M)$ satisfying $\eta|_{\pl M}=0$. We want to show that, up to reducing $\delta>0$, we have $\eta\equiv0$.

We claim that, for each $\eps>0$, there is $\delta>0$ small enough such that the diffeomorphism $\phi$ as above satisfies $\|\phi^{-1}-{\rm Id}\|_{C^2(M)}<\eps$. Indeed, if we set $q:=g_1-g_2$, then $(\phi^{-1})^*g_1=e^{-2\eta\circ \phi^{-1}}(g_1-q)$. Therefore, the map $\phi^{-1}:M\to M$ is  quasi-conformal, equal to the identity on the boundary $\partial M$, and with Beltrami coefficient $\mu_0\in C^\infty(M; \kappa^{-1}\otimes \overline{\kappa})$ having norm 
$\|\mu_0\|_{C^2(M)}\leq C\|q\|_{C^2(M)}$ for some $C>0$ depending only on $(M,g_1)$. Here, $\kappa$ denotes the canonical bundle of $(M,g_1)$. 
By \cite[Section 3]{Earle:1970oa}, the conformal structure of the interior of $(M,g_1)$ can be realized as $\Gamma\backslash\hh^2$ for some convex co-compact  Fuchsian group $\Gamma\subset {\rm PSL}(2,\rr)$, where $\hh^2=\{z\in\cc\ |\ {\rm Im}(z)>0\}$ is the hyperbolic space. We denote by $\Omega\subset \rr\cup \{\infty\}$ the set of discontinuity of $\Gamma$. We can assume that $\Omega$ contains  $\{0,1,\infty\}$. The discrete group $\Gamma$ acts properly discontinuously on $\hh^2\cup \Omega$ by holomorphic transformations. The universal cover $\til{M}$ can be identified with $\hh^2\cup\Omega$, and the fundamental group $\pi_1(M)$ can be identified with $\Gamma$. There is a holomorphic covering map $\pi: \hh^2\to M^\circ$, and the pull-back $\pi^*$ is a homeomorphism of the smooth conformal structures on $M$ to the set of $\Gamma$-equivariant smooth Beltrami differentials on $\hh^2\cup\Omega$.
The Beltrami differential $\mu_0$ lifts to a smooth Beltrami differential  on $\hh^2\cup\Omega$ which is $\Gamma$-equivariant. The continuity theorem in \cite[Section 8]{Earle:1970oa} implies that, for each Beltrami differential $\mu$ on $\hh^2\cup\Omega$ with $\|\mu\|_{L^\infty}<1$, there is a unique smooth quasiconformal map $\Phi_{\mu}:\hh^2\cup \Omega\to \hh^2\cup \Omega$ that fixes $\{0,1,\infty\}$,  satisfies the Beltrami equation 
$\pl_{\bar{z}}\Phi_{\mu}=\mu\pl_z \Phi_{\mu}$,
and such that the map $\mu\mapsto \Phi_{\mu}$ is continuous; here, we have equipped the domain of this map with the uniform 
$C^2$-topology on compact sets, and its codomain with the uniform 
$C^2$-topology on compact sets. We recall that $\Phi_{\mu}$ is a solution of the Beltrami equation if and only if it is a conformal transformation of $\hh^2\cup\Omega$ equipped with the conformal structure $|dz+\mu d\bar{z}|^2$ to $\hh^2\cup\Omega$ equipped with the conformal structure induced by $\cc$. By applying this result to the lift $\til{\mu}_0$ of the Beltrami differential $\mu_0$ to $\hh^2\cup \Omega$, we 
see by uniqueness that $\Phi_{\til{\mu}_0}$ is equal to the lift of $\phi^{-1}$ to $\hh^2\cup\Omega$. The group $\Gamma$ has a compact fundamental domain in $\hh^2\cup \Omega$, and therefore uniform $C^2$ estimates on $M$ follow from $C^2$ estimates on compact sets of $\hh^2\cup \Omega$. We deduce that $\|\phi^{-1}-{\rm Id}\|_{C^2(M)}\leq C\|\mu_0\|_{C^2(M)}$, which implies our claim.

We recall that any map $M\to M$ that is sufficiently $C^0$-close to the identity and fixes the boundary $\partial M$ is actually homotopic to the identity through maps that fix the boundary $\partial M$. Therefore, up to choosing $\epsilon>0$ small enough above, the diffeomorphism $\phi$ is homotopic to the identity through maps that fix the endpoints. If $\alpha$ is any geodesic of $(M,e^{2\eta}g_1)$ that joins two endpoints in $\partial M$, the curve $\phi\circ\alpha$ is the corresponding geodesic of $(M,g_2)$ joining the same endpoints. Since $\alpha$ and $\phi\circ\alpha$ are homotopic with fixed endpoints, and since $g_1$ and $g_2$ have the same marked boundary distance, $g_1$ and $e^{2\eta}g_1=\phi^*g_2$ have the same marked boundary distance as well. By Lemma~\ref{croke}, we conclude that $\eta=0$.
\end{proof}

\section{The negative curvature case using the method of Otal}

\subsection{The Liouville measure}

Let $(M,g)$ be a compact surface with strictly convex boundary and no conjugate points and, for now, 
we simply assume that the trapped set $K$ has zero Liouville measure, i.e.\ $\mu_g(K)=0$. 
We denote by $\til g=\pi^*g$ the lift of the Riemannian metric to the universal covering $\pi:\til M\to M$. 
By \eqref{muK} we have that $\mu_g(\Gamma_-\cup \Gamma_+)=0$, and thus  the tails 
$\til\Gamma_{\pm}\subset S\til M$ of the geodesic flow of $(\til M,\til g)$ 
have zero Liouville measure as well. Lemma~\ref{geodesics} implies that there exists a unique geodesic joining each pair of points $x,x'\in M$  
in each homotopy class of curves with endpoints $x$ and $x'$. Equivalently, for each pair of points $x,x'\in \til{M}$, there is a unique geodesic $\alpha_{\til g}(x,x'):[0,d_{\til{g}}(x,x')]\to \til M$ going from $x$ to $x'$. We consider the space 
\[\mc{G}:=(\partial \til{M}\x\partial \til{M})\setminus {\rm diag},\]
which we will see as the space of geodesics on $(\til M,\til g)$ by means of the map $(x,x')\mapsto\alpha_{\til g}(x,x')$. We denote by $\mc{M}$ the space of Borel measures on $\mc{G}$ invariant by the involution $(x,x')\mapsto(x',x)$.

Consider the open set
\[ U:=\big\{ (x,x',t)\in \mc{G}\x (0,\infty)\ \big|\ 0<t< d_{\til{g}}(x,x')\big\},\]
and the diffeomorphism
\begin{equation*}
\psi:  U\to S\til{M}\setminus (\til{\Gamma}_+\cup \til{\Gamma}_-) , \qquad 
\psi(x,x',t)=(\alpha_{\til g}(x,x')(t), \partial_t\alpha_{\til g}(x,x')(t)).
 \end{equation*} 
We denote by $\til\lambda$ the contact form on $S\til M$. We recall that the Liouville measure $\mu_{\til g}=|\til\lambda\wedge d\til\lambda|$ is invariant by the geodesic flow $\til{\varphi}_t$ of $(\til M,\til g)$. If $\til X$ denotes the geodesic vector field on $S\til M$, we have $\psi_*\pl_t=\til X$, and therefore
\begin{align}\label{needed_for_Liouville_1}
i_{\pl_t}(\psi^*(\til\lambda\wedge d\til\lambda))=\psi^*(i_{\til X}(\til\lambda\wedge d\til\lambda))=\psi^*(d\til\lambda).
\end{align}
Moreover, since $i_{\til X}d\til\lambda=0$, we have
\begin{align}\label{needed_for_Liouville_2}
\mc{L}_{\pl_t} \psi^*(d\til\lambda) = d( i_{\pl_t}\psi^*(d\til\lambda) ) = d( \psi^*i_{\til X}d\til\lambda) =0.
\end{align}
Equations \eqref{needed_for_Liouville_1} and \eqref{needed_for_Liouville_2} allow to define a measure $\eta_{\til{g}}\in\mc{M}$ satisfying
\begin{align}\label{definition_Liouville_on_geodesics}
\psi^*d\mu_{\til g}=d\eta_{\til g}\otimes dt.
\end{align}
With a slight abuse of terminology, we will call $\eta_{\til{g}}$ the \emph{Liouville measure} on $\mc{G}$ associated with the Riemannian metric $\til g$. Let us provide a useful expression of the Liouville measure in a particular coordinate system. For each pair of distinct points $x,x'\in\til M$, we denote by $\mc{F}(x,x')\subset\mc{G}$ the open subset of those $(y,y')\in\mc{G}$ whose associated geodesic $\alpha_{\til g}(y,y')$ has a positive transverse intersection with $\alpha_{\til g}(x,x')$. Notice that $\mc{F}(x,x')$ does not depend on the Riemannian metric $\til g$ whenever $x$ and $x'$ belong to the boundary $\partial\til M$. Consider the open set
\[V:=\big\{(\tau,\theta)\in(0,d_{\til{g}}(x,x'))\x (0,\pi)\ \big|\ \ell_{\til{g}}(\alpha(\tau),R_{\theta}\dot \alpha(\tau))+\ell_{\til{g}}(\alpha(\tau), -R_{\theta}\dot \alpha(\tau))<\infty\big\},\] 
where $\alpha(\tau):=\alpha_{\til g}(x,x')(\tau)$, and $R_\theta$ denotes the $+\theta$ rotation in the fibers of $S\til{M}$. We define a diffeomorphism 
\[\phi: V  \to \mc{F}(x,x'), \qquad \phi(\tau, \theta):=(y,y'),\]
where $y,y'$ are the endpoints of the geodesic passing trough $\alpha(\tau)$ and tangent to $R_{\theta}\dot \alpha(\tau)$:
\begin{align*}
 y'&=\pi_0\circ\til{\varphi}_{\ell_{\til{g}}(\alpha(\tau), R_{\theta}\dot \alpha(\tau))}(\alpha(\tau), R_{\theta}\dot \alpha(\tau) ),\\
 y&= 
\pi_0\circ\til{\varphi}_{\ell_{\til{g}}(\alpha(\tau), -R_{\theta}\dot \alpha(\tau))}(\alpha(\tau), -R_{\theta}\dot \alpha(\tau)).
\end{align*}
Here, $\pi_0:S\til M\to\til M$ denotes the base projection.

\begin{lemm}\label{Santalo}
The open set $V$ has full measure in $(0,d_{\til{g}}(x,x'))\x (0,\pi)$, and 
\begin{equation}\label{phi^*muL}
\phi^*\eta_{\til g}=\sin(\theta)d\tau\,d\theta.
\end{equation}
\end{lemm}

\begin{proof}
Let $W$ be the open set of points $(\tau,\theta,t)\in V\times\rr$ such that $\til{\varphi}_{t}(\alpha(\tau),R_\theta\dot \alpha(\tau))$ belongs to the interior of $S\til M$. We define the maps
\begin{align*}
\psi_1:W\to S\til M,&\qquad
\psi_1(\tau,\theta,t):=\til\varphi_{t}(\alpha(\tau),R_\theta\dot \alpha(\tau)),\\
\psi_2:W \to U,&\qquad
\psi_2(\tau,\theta,t):=(\phi(\tau,\theta),t+\ell_{\til{g}}(\alpha(\tau),-R_\theta\dot \alpha(\tau))),
\end{align*}
which are diffeomorphisms onto their images. Notice that $\psi_1=\psi\circ\psi_2$.

The Liouville measure $\mu_{\til g}$ induces a measure on the restriction $S\til M|_{\alpha}$ of the unit tangent bundle along the geodesic $\alpha=\alpha_{\til g}(x,x')$, given by
\begin{align*}
d\mu_{\til g,\alpha}(y,v):= 
|\til g_y(v,R_{\pi/2}\dot \alpha(\tau_y))|
\,d\mathrm{vol}_{\alpha}(y)\,dS_y(v),
\end{align*}
where $\tau_y\in(0,d_{\til g}(x,x'))$ is such that $\alpha(\tau_y)=y$, $d\mathrm{vol}_{\alpha}$ is the Riemannian volume induced by $\til g$ on $\alpha$, and $dS_y$ is the volume measure on the fiber $S_y\til M$ induced by $\til g$. We set $Z:=\psi_1(W)$. By Santal\'o's formula, for each $f\in C_c^\infty(Z)$  we have
\begin{align*}
\int_{Z} f\,d\mu_{\til g} 
&= 
\int_{S\til M|_{\alpha}} \int_{-\ell_{\til{g}}(y,-v)}^{\ell_{\til{g}}(y,v)}f(\til{\varphi}_t(y,v)) 
\, dt\, d\mu_{\til g,\alpha}(y,v) \\
&= 
\int_{0}^{d_{\til{g}}(x,x')} \int_{0}^\pi \int_{-\ell_{\til{g}}(\alpha(\tau),-R_\theta\dot \alpha(\tau))}^{\ell_{\til{g}}(\alpha(\tau),R_\theta\dot \alpha(\tau))}f(\til{\varphi}_t(\alpha(\tau),R_\theta\dot \alpha(\tau))) 
\sin(\theta) \, dt\, d\theta\,  d\tau ,
\end{align*}
and therefore $\psi_1^* d\mu_{\til g}=\sin(\theta)d\tau\,d\theta\,dt$. This, together with equation~\eqref{definition_Liouville_on_geodesics}, implies that $\psi_2^*(\eta_{\til g}\otimes dt)= \sin(\theta)d\tau\,d\theta\,dt$, which in turn implies~\eqref{phi^*muL}.
Since the incoming and outgoing tails $\til{\Gamma}_\pm$ have zero Liouville measure, another application of Santal\'o's formula along the same line as above implies that $V$ has full measure in $(0,d_{\til{g}}(x,x'))\x (0,\pi)$.
\end{proof}

An immediate consequence of Lemma \ref{Santalo} is that
\begin{equation}\label{moregeneral}
\eta_{\til g}(\mc{F}(x,x'))=\int_{0}^\pi\int_{0}^{d_{\til{g}}(x,x')}\sin(\theta) d\tau\, d\theta=2d_{\til{g}}(x,x'),\qquad \forall x,x'\in\til M.
\end{equation} 
In particular, the distance function $d_{\til{g}}$ is completely determined by (and actually equivalent to) the knowledge of the Liouville measure of every set $\mc{F}(x,x')$. In Otal's terminology \cite{Otal:1990ta, Otal:1990zh}, the quantity $\eta_{\til g}(\mc{F}(x,x'))$ is the \emph{intersection number} of the Liouville measure  with the Dirac measure supported on the geodesic $\alpha_{\til g}(x,x')$. These intersection numbers allow to recover the Liouville measure, as it follows from the next statement.

\begin{lemm}\label{douady}
Two measures $\mu,\mu'\in \mc{M}$ are such that $\mu(\mc{F}(x,x'))=\mu'(\mc{F}(x,x'))$ for all 
distinct points $x,x'\in \pl\til{M}$ if and only if $\mu=\mu'$.
\end{lemm}

\begin{proof} 
We recall that the boundary $\pl\til{M}$ is homeomorphic to a countable union of real lines embedded in the circle $S^1$. We consider two subsets $E,F\subset \pl\til{M}$, each one being the intersection of $\pl\til{M}$ with a compact interval in  $S^1$. We set 
\[
\mc{G}_{E,F}
:=
(E\times F)\setminus\mathrm{diag}.
\] 
All we have to prove is that $\mu(\mc{G}_{E,F})=\mu'(\mc{G}_{E,F})$. The subsets $E$ and $F$ may overlap. By the additivity property of measures, it is enough to consider two cases: $E=F$ or $E\cap F=\varnothing$.

\begin{figure}
\begin{center}
\begin{small}
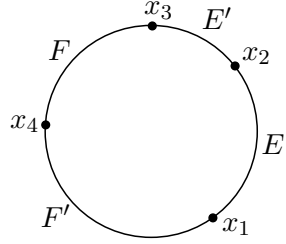 
\caption{The circle containing $\partial\til M$. The sets $E,E',F,F'$ are not actually intervals in the circle, but are the intersection of the intervals drawn with $\partial\til M$.}
\label{f:douady}
\end{small}
\end{center}
\end{figure}

Let us first consider the case where $E$ and $F$ are disjoint, or intersects at most at one boundary point. We denote by $x_1,x_2$ the points in the boundary of $E$ ordered according to the orientation of $E$, and analogously by $x_3,x_4$ the points in the boundary of $F$. We denote by $E',F'\subset\pl\til{M}$ the complementary regions (which may be empty), as in Figure~\ref{f:douady}. We have
\begin{align*}
\mu(\mc{F}(x_1,x_3)) & = \mu(\mc{G}_{E,F}) + \mu(\mc{G}_{E,F'})+\mu(\mc{G}_{E',F}) + \mu(\mc{G}_{E',F'}),\\
\mu(\mc{F}(x_4,x_2)) & = \mu(\mc{G}_{E,F}) + \mu(\mc{G}_{E,E'})+\mu(\mc{G}_{F',F}) + \mu(\mc{G}_{F',E'}),\\
\mu(\mc{F}(x_2,x_3)) & = \mu(\mc{G}_{E',E}) + \mu(\mc{G}_{E',F'})+\mu(\mc{G}_{E',F}),\\
\mu(\mc{F}(x_1,x_4)) & = \mu(\mc{G}_{E,F'}) + \mu(\mc{G}_{E',F'})+\mu(\mc{G}_{F,F'}).
\end{align*}
Since $\mu$ is invariant under the involution $(x,x')\mapsto(x',x)$, we have $\mu(\mc{G}_{E',F})=\mu(\mc{G}_{F,E'})$, $\mu(\mc{G}_{E',F'})=\mu(\mc{G}_{F',E'})$, and $\mu(\mc{G}_{F',F})=\mu(\mc{G}_{F,F'})$. Therefore
\begin{align*}
\mu(\mc{G}_{E,F})& =\tfrac12 \big( \mu(\mc{F}(x_1,x_3)) + \mu(\mc{F}(x_4,x_2)) - \mu(\mc{F}(x_2,x_3)) - \mu(\mc{F}(x_1,x_4)) \big)\\
& =\tfrac12 \big( \mu'(\mc{F}(x_1,x_3)) + \mu'(\mc{F}(x_4,x_2)) - \mu'(\mc{F}(x_2,x_3)) - \mu'(\mc{F}(x_1,x_4)) \big)\\
&= \mu'(\mc{G}_{E,F}).
\end{align*}

Now, consider the case where $E=F$. We parametrize the circle $S^1$ containing $\pl\til{M}$ from $\theta=0$ to $\theta=2\pi$ in such a way that $E=[0,\theta_0]\cap \pl\til{M}$ for some $\theta_0\in[0,2\pi]$. We split $E$ in two parts of equal size, so that $E=E_1\cup E_2$ for $E_1=[0,\theta_0/2]\cap \pl\til{M}$ and $E_2=[\theta_0/2,\theta_0]\cap \pl\til{M}$. We apply the same splitting iteratively, so that $E_1=E_{11}\cup E_{12}$, $E_2=E_{21}\cup E_{22}$, and more generally $E_{i_0...i_{n}}=E_{i_0...i_{n-1}1}\cup E_{i_0...i_{n-1}2}$. Notice that
\begin{align*}
\mc{G}_{E,E}=\bigcup_{n=1}^\infty \bigcup_{i_1,...,i_{n-1}\in\{1,2\}} 
\big(
\mc{G}_{E_{i_0...i_{n-1}1},E_{i_0...i_{n-1}2}}
\cup
\mc{G}_{E_{i_0...i_{n-1}2},E_{i_0...i_{n-1}1}}
\big).
\end{align*}
By the $\sigma$-additivity of measures, we have
\begin{align*}
\mu(\mc{G}_{E,E})
&=\sum_{n=1}^\infty \sum_{i_1,...,i_{n-1}\in\{1,2\}} 
\big(
\mu(\mc{G}_{E_{i_0...i_{n-1}1},E_{i_0...i_{n-1}2}})
+
\mu(\mc{G}_{E_{i_0...i_{n-1}2},E_{i_0...i_{n-1}1}})
\big)\\
&=\sum_{n=1}^\infty \sum_{i_1,...,i_{n-1}\in\{1,2\}} 
\big(
\mu'(\mc{G}_{E_{i_0...i_{n-1}1},E_{i_0...i_{n-1}2}})
+
\mu'(\mc{G}_{E_{i_0...i_{n-1}2},E_{i_0...i_{n-1}1}})
\big)\\
&=\mu'(\mc{G}_{E,E}).
\qedhere
\end{align*}
\end{proof}

The following is an immediate consequence of equation~\eqref{moregeneral} and Lemma~\ref{douady}.

\begin{corr}\label{same_dtilde_same_Liouville}
Let $M$ be a compact surface with smooth boundary, and $g_1,g_2$ two  Riemannian metrics with no conjugate points, strictly convex boundary $\partial M$, and 
trapped set of zero Liouville measure. Then $g_1$ and $g_2$ have the same marked boundary distance if and only if they have the same Liouville measure $\eta_{\til g_1}=\eta_{\til g_2}$.
\hfill\qed
\end{corr}

\subsection{The rigidity result}
Let $M$ be a compact surface with smooth boundary. On $M$, we consider a negatively-curved Riemannian metric $g_1$ that makes $\partial M$ strictly convex; 
we then consider  another Riemannian metric $g_2$ with no conjugate points, that makes $\partial M$ strictly convex, and such that the trapped set of its geodesic flow has zero Liouville measure. 
Since $(M,g_1)$ has negative curvature, it does not have conjugate points and the trapped set of its geodesic flow has zero Liouville measure (see \cite[Section~2.4]{Guillarmou:2014jk}). Therefore, $g_1$ and $g_2$ satisfy the assumptions of Corollary~\ref{same_dtilde_same_Liouville}. Assume that they also have the same marked boundary distance $d_{\til g_1}=d_{\til g_2}$, and thus the same Liouville measures $\eta_{\til g_1}=\eta_{\til g_2}$ according to Corollary~\ref{same_dtilde_same_Liouville}. We denote by $SM$ and 
$S\til{M}$ the unit tangent bundle for $g_1$ over $M$ and for $\til{g}_1$ over $\til{M}$, and by $\til\Gamma_{\pm}$ the incoming/outgoing tails of the geodesic flow of $(\til M,\til g_1)$.
For each $(x,v)\in S\til{M}$, let $R_\theta v\in S_x\til{M}$ be the vector obtained after a counterclockwise rotation of  $v$ in the fiber by an angle $\theta\in[0,\pi]$, where the angle is measured by $\til g_1$. Consider the two geodesics of $\til{g}_1$ of $S\til{M}$ intersecting at $x$ and with tangent vectors $v$ and $R_\theta v$. If both $(x,v)$ and $(x,R_\theta v)$ are not in $\til{\Gamma}_-\cup\til{\Gamma}_+$, these geodesics 
can be written as $\alpha_{\til g_1}(z,z')$ and $\alpha_{\til g_1}(w,w')$ for some $z,z',w,w'\in \pl\til{M}$. 
The corresponding geodesics $\alpha_{\til g_2}(z,z')$ and $\alpha_{\til g_2}(w,w')$ for the metric $\til{g}_2$ intersect at some point $x''$ with an angle $\til\theta''(x,v,\theta)$.
This defines a function 
\[\til\theta'' : W_1 \to [0,\pi],\]
where $W_1$ is the open set 
\[W_1:=\big\{(x,v,\theta)\in S\til{M}\x[0,\pi]\ \big|\ (x,v), 
(x,R_\theta v)\notin (\til{\Gamma}_-\cup\til{\Gamma}_+)\big\}.\]
The complement $(S\til{M}\x[0,\pi])\setminus W_1$ has zero measure with respect to $\mu_{\til g_1}\otimes d\theta$. Let  $Z\subset \mc{G}\times \mc{G}$ be the open set of quadruples $(z,z',w,w')$ whose associated geodesics $\alpha_{\til g_1}(z,z')$ and $\alpha_{\til g_1}(w,w')$ are distinct and intersect at one point. We define the diffeomorphism
\begin{equation*}
\kappa_{1}:W_1\to Z,\qquad \kappa_{1}(x,v,\theta)= (z,z',w,w'),
\end{equation*} 
where, as before, $\alpha_{\til g_1}(z,z')$ and $\alpha_{\til g_1}(w,w')$ are the geodesic passing through the point $x$ with tangent vectors $v$ and $R_\theta v$ respectively. For the Riemannian metric $\til g_2$ we can define an analogous diffeomorphism $\kappa_{2}:W_2\to Z$. Notice that the function $\til\theta''$ appears in the composition $\kappa_{2}^{-1}\circ\kappa_1$, which indeed has the form
\begin{align}\label{composition}
\kappa_{2}^{-1}\circ\kappa_1(x,v,\theta)=(x''(x,v,\theta),v''(x,v,\theta),\til\theta''(x,v,\theta)).
\end{align}
Therefore $\til\theta''$ is a smooth function.  

\begin{lemm}\label{lemma_theta}
For all $(x,v)\in S\til{M}$ and $\theta_1,\theta_2\in[0,\pi]$ such that $\theta_1+\theta_2\in [0,\pi]$, $(x,v,\theta_1)\in W_1$, and $(x,R_{\theta_1}v,\theta_2)\in W_1$, we have
\begin{gather}
\label{symmetry_theta_prime}
\til\theta''(x,v,\theta_1)+\til\theta''(x,R_{\theta_1} v, \pi-\theta_1)=\pi,\\
\label{subadditivity}
\til\theta''(x,v,\theta_1)+\til\theta''(x,R_{\theta_1}v,\theta_2)
\leq 
\til\theta''(x,v,\theta_1+\theta_2).
\end{gather}
Moreover, if we set 
\begin{equation}\label{geodesics_passing_through_x}
\begin{split}
(z,z',w,w')& := \kappa_{1}(x,v,\theta_1),\\
(w,w',y,y') & := \kappa_{1}(x,R_{\theta_1} v,\theta_2),
\end{split}
\end{equation}
the inequality~\eqref{subadditivity} is satisfied as an equality if and only if the three geodesics $\alpha_{\til g_2}(z,z')$, $\alpha_{\til g_2}(w,w')$, and $\alpha_{\til g_2}(y,y')$ intersect at one point of $\til M$.
\end{lemm}

\begin{proof}
Equation~\eqref{symmetry_theta_prime} follows from the very definition of $\til\theta''$. Consider the three pairs $(z,z'),(w,w'),(y,y')\in\mc{G}$  defined by~\eqref{geodesics_passing_through_x}, and  the corresponding geodesics $\alpha_{\til g_2}(z,z')$, $\alpha_{\til g_2}(w,w')$, and $\alpha_{\til g_2}(y,y')$ for the Riemannian metric $\til g_2$. Some portions of these three geodesics are the edges of a geodesic triangle in $(\til{M},\til g_2)$ whose vertices are the mutual intersections among $\alpha_{\til g_2}(z,z')$, $\alpha_{\til g_2}(w,w')$, and $\alpha_{\til g_2}(y,y')$. This triangle may degenerate to a single point $x_0$ if the three geodesics intersect at $x_0$. The interior angles of this geodesic triangle are precisely $\til\theta''(x,v,\theta_1)$, $\til\theta''(x,R_{\theta_1}v,\theta_2)$, and $\til\theta''(x,R_{\theta_1+\theta_2}v,\pi-\theta_1-\theta_2)$. 
Since $\til{g}_2$ has negative curvature, Gauss-Bonnet formula for geodesic polygons implies 
\begin{align}
\label{GB}
\til\theta''(x,v,\theta_1)+\til\theta''(x,R_{\theta_1}v,\theta_2)+\til\theta''(x,R_{\theta_1+\theta_2}v,\pi-\theta_1-\theta_2)
\leq \pi.
\end{align}
This inequality is not strict if and only if the geodesic triangle is reduced to a point $x_0\in\til M$, that is, if and only if the three geodesics $\alpha_{\til g_2}(z,z')$, $\alpha_{\til g_2}(w,w')$, and $\alpha_{\til g_2}(y,y')$ intersect at $x_0$. Finally, \eqref{GB} and \eqref{symmetry_theta_prime} imply~\eqref{subadditivity}.
\end{proof}

\begin{proof}[Proof of Theorem~\ref{Th2}]
Notice that $\til\theta''$ descends to $SM \x[0,\pi]$ as a measurable bounded map $\theta'':SM\x[0,\pi]\to[0,\pi]$ which is smooth on an open set of full measure.
We introduce the average angle function
\[ \Theta : [0,\pi]\to [0,\pi], \qquad \Theta(\theta):=\frac{1}{{\rm Vol}_{g_1}(SM)}\int_{SM} \theta''(x,v,\theta)\,d\mu_{g_1}(x,v),
\]
which is continuous by Lebesgue theorem. Since $\theta''(x,v,0)=0$ and $\theta''(x,v,\pi)=\pi$ almost everywhere, we have 
\begin{align}
\label{endpoints_Theta}
\Theta(0)=0,\qquad \Theta(\pi)=\pi.
\end{align}
Notice that the rotation $R_\theta$, seen as a diffeomorphism of the unit tangent bundle $SM$, preserves the Liouville measure $\mu_{g_1}$. Clearly, Equations~\eqref{symmetry_theta_prime} and~\eqref{subadditivity} still hold if we replace $\til\theta''$ with $\theta''$, and if we integrate them on $SM$ against $\mu_{g_1}$ we obtain 
\begin{align}
\label{symmetry_Theta}
\Theta(\pi-\theta)&=\pi-\Theta(\theta),&&\forall\theta\in[0,\pi],\\
\label{subadditivity_Theta}
\Theta(\theta_1+\theta_2)&\geq \Theta(\theta_1)+\Theta(\theta_2),&&\forall\theta_1,\theta_2\in[0,\pi]\mbox{ with }\theta_1+\theta_2\in [0,\pi].
\end{align}
If $f:[0,\pi]\to\rr$ is a continuous convex function, we have Jensen's inequality
\begin{align}
\label{Jensen}
f(\Theta(\theta)) \leq \frac{1}{{\rm Vol}_{g_1}(SM)}\int_{SM} f(\theta''(x,v,\theta))\,d\mu_{g_1}(x,v). 
\end{align}
Actually, if $f$ is strictly convex, this inequality is satisfied as an equality if and only if $\theta''(x,v,\theta)=\theta$ for almost all $(x,v)\in SM$. Integrating in $\theta$, by Fubini's Theorem we get
\[ \int_0^\pi f(\Theta(\theta))\sin(\theta)d\theta 
\leq \frac{1}{{\rm Vol}_{g_1}(SM)}\int_{SM} \int_0^\pi f(\theta''(x,v,\theta))\sin(\theta)d\theta\,d\mu_{g_1}(x,v).\]
We set 
\begin{align*}
F(x,v):=\int_0^\pi f(\theta''(x,v,\theta))\sin(\theta)d\theta.
\end{align*}
By applying Santal\'o's formula, we obtain
\[ 
\int_{SM}F(x,v)\,d\mu_{g_1}
=
\int_{\pl_-SM}\int_{0}^{\ell_{g_1}(x,v)} F(\varphi_t(x,v))dt\,d\mu_{g_1,\nu}(x,v).
\]
We fix two arbitrary distinct points $x,x'\in\partial\til M$. For $i=1,2$, Lemma~\ref{Santalo} gives a subset of full measure $V_i\subset(0,d_{\til g_1}(x,x'))\times(0,\pi)$ and a diffeomorphisms $\phi_i:V_i\to\mc{F}(x,x')$ such that $\phi_i^*\eta_{\til g_i}=\sin(\theta)d\tau\,d\theta$. We consider the composition $\phi_2^{-1}\circ \phi_1: V_1\to V_2$, which has the form
\[\qquad   (\tau, \theta)\mapsto ( \tau''(\tau,\theta), 
\til\theta''(\alpha(\tau),\dot\alpha(\tau),\theta)),\]
where $\alpha(\tau):=\alpha_{\til g_1}(x,x')$. Since $\eta_{\til g_1}=\eta_{\til g_2}$, Lemma~\ref{Santalo} implies
\begin{align*}
(\phi_2^{-1}\circ \phi_1)^* \sin(\theta) d\tau\,d\theta=
\phi_1^*\eta_{\til g_2}=
\phi_1^*\eta_{\til g_1}=
\sin(\theta) d\tau\,d\theta.
\end{align*}
From this we infer that, for all $(x,v)\in \pl_-S\til{M}\setminus \til{\Gamma}_-$,
\begin{align*}
\int_{0}^{\ell_{g_1}(\pi(x,v))}F(\varphi_\tau(\pi(x,v)))d\tau= &
\int_{0}^{d_{\til{g}_1}(x,x')}F(\pi(\til{\varphi}_\tau(x,v)))d\tau \\
=&\int_{0}^{d_{\til{g}_1}(x,x')}\int_0^\pi 
f(\til\theta''(\alpha(\tau),\dot\alpha(\tau),\theta))\sin(\theta)d\theta\, d\tau \\
 =&\int_{0}^{d_{\til{g}_1}(x,x')}\int_0^\pi 
f(\theta)\sin(\theta)d\theta\,d\tau\\ 
=& \ell_{g_1}(\pi(x,v))\int_0^\pi 
f(\theta)\sin(\theta)d\theta,
\end{align*}
where $v\in S_{x}\til M$ denotes the tangent vector to the geodesic $\alpha_{\til g_1}(x,x')$ at the starting point $x$. We integrate the previous equality over $\partial_-SM$ against $d\mu_{g_1,\nu}$. By applying Santal\'o's formula to the left-hand side, the volume computation~\eqref{Volume} to the right-hand side, and finally Jensen's inequality~\eqref{Jensen} to the left-hand side, we obtain
\begin{equation*}
\int_0^\pi f(\Theta(\theta))\sin(\theta)d\theta 
\leq \int_0^\pi f(\theta)\sin(\theta)d\theta.
\end{equation*}
Since this inequality must be satisfied for all continuous convex functions $f:[0,\pi]\to\rr$, Equations~\eqref{endpoints_Theta}, \eqref{symmetry_Theta}, and \eqref{subadditivity_Theta} readily imply that $\Theta(\theta)=\theta$ for all $\theta\in[0,\pi]$, see \cite[Lemma 8]{Otal:1990ta}. Therefore, the inequality~\eqref{subadditivity} is satisfied as an equality. By the ``moreover'' part of Lemma~\ref{lemma_theta}, whenever three geodesics $\alpha_{\til g_1}(x,x')$, $\alpha_{\til g_1}(y,y')$, $\alpha_{\til g_1}(z,z')$ intersect at a single point $x_0$, the corresponding geodesics for the other Riemannian metric $\alpha_{\til g_2}(x,x'),\alpha_{\til g_2}(y,y'),\alpha_{\til g_2}(z,z')$ must intersect in a single point, which we denote by $\til\psi(x_0)$, as well. This defines a map $\til\psi:\til M\to\til M$.  Fix a point $x\in\til M$. The domain $W_1$ of $\kappa_1:W_1\to Z$ is an open set with projection on the base given by the whole manifold $\til{M}$. Consider the composition of diffeomorphisms $\kappa_2^{-1}\circ\kappa_1$, which has the form~\eqref{composition}. For each $x\in U$, if $v\in T_x\til M$ and $\theta\in[0,\pi]$ are such that both $(x,v)$ and $(x,R_\theta v)$ are not in $\til{\Gamma}_-\cup\til{\Gamma}_+$, we have $\til\psi(x)=x''(x,v,\theta)$. This shows that  $\til\psi$ is smooth. 

For $x_0\in \pl\til{M}$, we can choose four points
$y,y',z,z'\in \pl\til{M}$ near $x_0$ so that, according to the orientation of $\partial \til M$, the points are in the order $y,z,x_0,y',z'$. If both $y,z'$ tend to $x_0$, then so do the points $y',z$ and, by the  strict convexity of $\pl\til{M}$, the intersection $x:=\alpha_{\til g_1}(y,y')\cap\alpha_{\til g_1}(z,z')$. By  the same reasoning $\til\psi(x)=\alpha_{\til g_2}(y,y')\cap\alpha_{\til g_2}(z,z')$  converges to $x_0$ as well. This proves that $\til\psi|_{\pl \til{M}}={\rm Id}$.

Clearly $\til\psi$ maps every geodesic of $\til g_1$ to the geodesic of $\til g_2$ with the same endpoints. For $i=1,2$ and for all distinct points $x,x'\in\til M$, consider the open set $\mc{F}_i(x,x')\subset\mc{G}$ of those $(y,y')$ whose associated geodesic $\alpha_{\til g_i}(y,y')$ has a positive transverse intersection with $\alpha_{\til g_i}(x,x')$. It is straighforward to see that \[\mc{F}_1(x,x')=\mc{F}_2(\til\psi(x),\til\psi(x')).\] Therefore, by Equation~\eqref{moregeneral}, we have
\begin{align*}
d_{\til{g}_1}(x,x')=\tfrac12\eta_{\til g_1}(\mc{F}_1(\til\psi(x),\til\psi(x')))=\tfrac12\eta_{\til g_2}(\mc{F}_2(\til\psi(x),\til\psi(x')))=d_{\til{g}_2}(\til\psi(x),\til\psi(x')).
\end{align*}
This readily implies that $\til\psi^*\til g_2=\til g_1$. Indeed, for all $(x,v)\in SM$, if $\zeta:[0,\epsilon)\to\til M$ is a smooth curve such that $\zeta(0)=x$ and $\dot\zeta(0)=v$, we have
\begin{align*}
\|v\|_{\til g_1} =
\lim_{t\to0} \frac{d_{\til{g}_1}(\zeta(0),\zeta(t))}{t}
=
\lim_{t\to0} \frac{d_{\til{g}_2}(\til\psi(\zeta(0)),\til\psi(\zeta(t)))}{t}
=
\|d\til\psi(x)v\|_{\til g_2}.
\end{align*}
Finally, since the fundamental group $\pi_1(M)$ acts on $\til M$ by isometries with respect to both $\til{g}_1$ and $\til{g}_2$, we have $\til\psi(\gamma\cdot x)=\gamma\cdot\til\psi(x)$ for each $\gamma\in \pi_1(M)$, and thus $\til\psi$ descends to a smooth isometry $\psi:M\to M$ fixing the boundary. 
\end{proof}

\bibliography{biblio}
\bibliographystyle{amsalpha}

\end{document}